\documentclass[11pt,twoside,english]{amsart}
\advance\oddsidemargin by -1.8cm
\advance\evensidemargin by -1.6cm
\advance\topmargin by -1.0cm 
\usepackage{amssymb}
\usepackage{babel}
\usepackage{amstext}
\usepackage{amscd}   
\usepackage{epsfig}  
\usepackage{rotating}
\usepackage{fancybox}

\theoremstyle{plain}
\newtheorem{cor}{Corollary}[section]
\newtheorem{lem}{Lemma}[section]
\newtheorem{thm}{Theorem}[section]            
\newtheorem{prop}{Proposition}[section]
\theoremstyle{definition}
\newtheorem{exa}{Example}[section]
\newtheorem{NB}{Remark}[section]
\newtheorem{dfn}{Definition}[section]

\newcommand{\bdm}{\begin{displaymath}}
\newcommand{\edm}{\end{displaymath}}
\newcommand{\be}{\begin{equation}}
\newcommand{\ee}{\end{equation}}
\newcommand{\ba}[1]{\begin{array}{#1}}
\newcommand{\ea}{\end{array}}

\newcommand{\btab}{\begin{tabular}}
\newcommand{\etab}{\end{tabular}}

\newcommand{\op}{\oplus}
\newcommand{\ox}{\otimes}

\newcommand{\Id}{\ensuremath{\mathrm{Id}}}
\newcommand{\tr}{\ensuremath{\mathrm{tr}}}

\newcommand{\del}{\partial}

\newcommand{\C}{\ensuremath{\mathbb{C}}}
\newcommand{\R}{\ensuremath{\mathbb{R}}}

\newcommand{\Z}{\ensuremath{\mathbb{Z}}}

\newcommand{\eps}{\ensuremath{\varepsilon}} 
\newcommand{\vphi}{\ensuremath{\varphi}}    

\newcommand{\End}{\ensuremath{\mathrm{End}}}

\newcommand{\Ric}{\ensuremath{\mathrm{Ric}}}

\newcommand{\Div}{\ensuremath{\mathrm{div}}}

\newcommand{\II}{\ensuremath{\mathrm{II}\,}}
\newcommand{\grad}{\ensuremath{\mathrm{grad}\,}}
\newcommand{\kr}{\ensuremath{\mathcal{R}}}
\newcommand{\cyclic}[1]{\stackrel{{\scriptsize #1}}{\mathfrak{S}}}
\newcommand{\kulk}{\ensuremath{  {~\wedge\!\!\!\!\!\bigcirc~}     }}

\newcommand{\Orth}{\ensuremath{\mathrm{O}}}

\begin{document}
\def\haken{\mathbin{\hbox to 6pt{%
                 \vrule height0.4pt width5pt depth0pt
                 \kern-.4pt
                 \vrule height6pt width0.4pt depth0pt\hss}}}
    \let \hook\intprod
\setcounter{equation}{0}
%
%
\thispagestyle{empty}
%
\date{\today}
\title
{Manifolds with vectorial torsion}
%
%
%
\author{Ilka Agricola}
\author{Margarita Kraus}
\address{\hspace{-5mm} 
Ilka Agricola\newline
Fachbereich Mathematik und Informatik \newline
Philipps-Universit\"at Marburg\newline
Hans-Meerwein-Strasse \newline
D-35032 Marburg, Germany\newline
{\normalfont\ttfamily agricola@mathematik.uni-marburg.de}}
\address{\hspace{-5mm} 
Margarita Kraus\newline
Fachbereich Physik, Mathematik und Informatik\newline
Johannes-Gutenberg-Universit\"at Mainz\newline
Staudingerweg 9\newline
D-55099 Mainz, Germany\newline
{\normalfont\ttfamily mkraus@mathematik.uni-mainz.de}}
%
\subjclass[2000]{Primary 53 C 25; Secondary 81 T 30}
\thanks{Ilka Agricola 
acknowledges financial support by the
DFG within the priority programme 1388 "Representation theory".}
\keywords{metric connection with vectorial torsion; curvature; Dirac operator; 
warped product; parallel spinor; Killing spinor; Weyl manifold; Weyl tensor}  
\begin{abstract}
The present note deals with the properties of metric connections $\nabla$ with
vectorial torsion $V$ on semi-Riemannian manifolds $(M^n,g)$. 
We show that the $\nabla$-curvature is symmetric if and only if $V^{\flat}$ is
closed, and that $V^\perp$ then defines an $(n-1)$-dimensional integrable
distribution on $M^n$. If the vector field $V$ is exact, we show that   the 
$V$-curvature coincides up to global rescaling with the Riemannian
curvature of a conformally equivalent metric. We prove 
that it is possible to construct connections with vectorial
torsion on warped products of arbitrary dimension matching a given
Riemannian or Lorentzian curvature---for example, a $V$-Ricci-flat connection
with vectorial torsion in dimension $4$, explaining some constructions
occurring in general relativity. Finally, 
we investigate the Dirac operator $D$ of
a connection with vectorial torsion. We prove that for exact vector fields, 
the $V$-Dirac spectrum coincides
with the spectrum of the Riemannian Dirac operator.  We investigate 
in detail the existence of $V$-parallel spinor fields; several examples are 
constructed. It is known that  the existence 
of a $V$-parallel spinor field implies $dV^\flat=0$ for $n=3$ or $n\geq 5$;
for $n=4$, this is only true  on compact manifolds. 
We prove an identity relating the  $V$-Ricci curvature to 
the curvature in the spinor bundle. This result allows us to prove that
if there exists a nontrivial $V$-parallel spinor, then $\Ric^V=0$ for
$n\neq 4$ and $\Ric^V(X)=X\haken dV^\flat$ for $n=4$. We  conclude 
that the manifold  is  conformally equivalent either  
to a manifold with Riemannian parallel spinor or to a manifold  whose universal 
cover is the product of $\R $ and an Einstein space of positive scalar 
curvature. We also prove that if $dV^\flat=0$, there are no non-trivial
$\nabla$-Killing spinor fields.
\end{abstract}
\maketitle
\pagestyle{headings}
%
%
\section{Introduction}\noindent
%

The present note deals with metric connections on
semi-Riemannian manifold $(M^n, g)$ of the form
\bdm
\nabla_X  Y = \nabla^g_X Y+ g(X,Y)V - g(V,Y)X, 
\edm
where $V$ denotes a fixed vector field on $M$ and $\nabla^g$ is the usual
Levi-Civita connection. This is one of the three basic types of metric 
connections introduced by \'Elie Cartan
(see Section \ref{sec.metric-conn}), and is called a metric connection
with \emph{vectorial torsion}. 
These connections are particularly interesting on surfaces, in as much
that \emph{every} metric connection on a surface is of this type.
By a seminal theorem of Ambrose and Singer \cite{Ambrose&S58}, a 
simply connected Riemannian manifold is homogeneous if and only if its
canonical connection $\nabla^c$ has $\nabla^c$-parallel torsion and
curvature. Cartan's result implies therefore that there are 
three possible basic 
types of non-symmetric  homogeneous spaces (see also \cite{Tricerri&V1}).
If the canonical connection has vectorial torsion, the condition
$\nabla^c$-parallel torsion is equivalent to $\nabla^c V=0$.
Such homogeneous spaces were intensively studied in the past, while only
few results on general semi-Riemannian manifolds equipped with a 
metric connection $\nabla$ with vectorial torsion are available. 
In \cite[Thm 5.2]{Tricerri&V1}, F.~Tricerri and L.~Vanhecke  showed
that if $M$ is connected, complete,  simply-connected and
$\nabla V= \nabla \kr = 0$, then $(M,g)$ has to be
isometric to hyperbolic space (see also \cite{CastrillonGS13}
for an alternative modern proof). Hence, it is too strong a condition
to require that $V$ is $\nabla$-parallel. 

Denote by $V^\flat$ the dual
$1$-form of $V$; we shall loosely call the vector field $V$ \emph{closed}
if $d V^\flat=0$. In Section \ref{sec.curvature}, we prove that $\nabla$ has 
symmetric curvature if and only if $V$ is closed,  and that this is equivalent
to the condition that the $(n-1)$-dimensional distribution $D:=V^\perp$ is 
involutive. We thus believe that $d V^\flat=0$ is a richer, geometrically 
interesting replacement of the condition $\nabla V=0$.
If the vector field is exact, we show that   the 
$V$-curvature coincides up to global rescaling with the Riemannian
curvature of a conformally equivalent metric. We then investigate more in 
detail the case of warped products, both with Riemannian and
Lorentzian signature. In particular, we prove 
that it is possible to construct connections with vectorial
torsion on warped products of arbitrary dimension matching a given
Riemannian or Lorentzian curvature---for example, a $V$-Ricci-flat connection
with vectorial torsion in dimension $4$. This explains the occurrence of some 
examples of $V$-Ricci-flat manifolds known in physics \cite{Hehl}.

In the last part of the paper, we investigate the Dirac operator $D$ of
a connection with vectorial torsion. As already observed by
Friedrich in \cite{Friedrich79}, the Dirac operator is not formally
self-adjoint anymore, thus making its analysis much harder. Nevertheless,
we prove that for exact vector fields, the $V$-Dirac spectrum coincides
with the spectrum of the Riemannian Dirac operator. Based on results from
\cite{Pfaeffle&S11}, we derive a formula of Schr\"odinger-Lichnerowicz type
formula for $D^* D$. We investigate in detail the existence of
$V$-parallel spinor fields (i.\,e.~parallel for the metric connection with vectorial torsion $V$); several examples are constructed on warped 
products. By a result of \cite{Mor}, it is known that  the existence 
of a $V$-parallel spinor field implies $dV^\flat=0$ for $n=3$ or $n\geq 5$;
for $n=4$, this is only true  on compact manifolds. To get a more detailed
picture, we prove an identity relating the  $V$-Ricci curvature to 
the curvature in the spinor bundle. This result allows us to prove that
if there exists a nontrivial $V$-parallel spinor, then $\Ric^V=0$ for
$n\neq 4$ and $\Ric^V(X)=X\haken dV^\flat$ for $n=4$---in particular, the
$V$-Ricci curvature is totally skew-symmetric in the latter case, a rather
unfamiliar situation. We  conclude that the manifold  is  
conformally equivalent either  
to a manifold with parallel spinor or to a manifold  whose universal 
cover is the product of $\R $ and an Einstein space of positive scalar 
curvature. We also prove that if $dV^\flat=0$, there are no non-trivial
$\nabla$-Killing spinor fields.
%
\section{Metric connections with torsion}\label{sec.metric-conn}
%
Consider a semi-Riemannian manifold $(M^n, g)$ of index $k$. The 
difference between its Levi-Civita connection $\nabla^g$ and any linear 
connection $\nabla$ is  a $(2,1)$-tensor field $A$,
\bdm 
\nabla_X Y\ =\ \nabla^g_X Y + A(X,Y),\quad X,Y \in TM^n.
\edm
Following Cartan, we study the algebraic 
types of the torsion tensor for a metric connection.
Denote by the same symbol the $(3,0)$-tensor
derived from a $(2,1)$-tensor  by contraction with the metric.
We identify $TM^n$ with $(TM^n)^*$ using $g$ from now on. 
Let $\mathcal{T}$ be the $n^2(n-1)/2$-dimensional space of all 
possible torsion tensors,
\bdm
\mathcal{T}\ =\ \{T\in\ox^3 TM^n \ | \ T(X,Y,Z)= - \, T(Y,X,Z) \}
\ \cong \ \Lambda^2 TM^n\ox TM^n \, .
\edm
A connection $\nabla$ is metric if and only if $A$ belongs to the space
\bdm
\mathcal{A}^g\ :=\ TM^n\ox(\Lambda^2 TM^n) \ = \ \{A \in\ox^3 TM^n \ | 
\ A(X,V,W) +  A(X,W,V) \ =\ 0\} \, .
\edm 
In particular, $\dim \mathcal{A}^g = \dim \mathcal{T}$, reflecting
the fact that metric connections can be uniquely characterized by their
torsion. The following proposition has been proven in  {\cite[p.51]{Cartan25a}, 
\cite{Tricerri&V1}} in the Riemannian case, but one easily checks that it
holds also for  semi-Riemannian  manifolds.
\begin{prop}
\label{classessum}
The spaces $\mathcal{T}$ and $\mathcal{A}^g$ are isomorphic
as $\Orth(n,k)$ representations, an  equivariant bijection being 
\bdm
T(X,Y,Z)  =  A(X,Y,Z)-A(Y,X,Z), \quad
2 \, A(X,Y,Z)  =  T(X,Y,Z)-T(Y,Z,X)+T(Z,X,Y).
\edm
%
%
For $n=2$,  
$\mathcal{T}\cong\mathcal{A}^g\cong\R^2$ is $\Orth(2,k)$-irreducible, while
for $n\geq 3$,  it splits under the action of $\Orth(n,k)$
into the sum of three irreducible representations,
\bdm
\mathcal{T}\cong TM^n \op \Lambda^3(M^n) \op \mathcal{T}'.
\edm
%
%
%

\end{prop}
The connection $\nabla$ is said to have  \emph{vectorial torsion} if
its torsion tensor lies in  the first space of the
decomposition in Proposition \ref{classessum}, i.\,e.~if it is essentially
defined by some vector field $V$ on $M$. The tensors $A$ and $T$
can then  be directly expressed through  $V$  as
\be\label{eq.torsion}
A_V(X)Y \ =\ g(X,Y)V-g(V,Y)X,\quad
T_V(X,Y,Z)\ =\ g\big(g( V, X)Y- g(V, Y)X,Z\big).
\ee
The connection $\nabla$ is said to have  \emph{skew-symmetric 
torsion} or just \emph{skew torsion} if its torsion tensor lies in  
the second component of the
decomposition in Proposition \ref{classessum}, i.\,e.~it is
given by a $3$-form. The third torsion component has no geometric interpretation.
Recall that homogeneous spaces whose canonical connection has
skew torsion are usually known as \emph{naturally reductive} homogeneous spaces;
their holonomy properties and their classification are currently
topics of great interest.
While metric connections with $\nabla$-parallel skew torsion have a rich
geometry (for example, the characteristic connections of Sasaki manifolds, of
nearly K\"ahler manifolds, and of nearly parallel $G_2$ manifolds have this 
property), metric connections with $\nabla$-parallel vectorial torsion
are  rare--the underlying manifold has to be covered by hyperbolic space.
Alas, this means that the general holonomy principle will not be applicable for the
investigation of metric connections with vectorial torsion.
%
\section{Curvature}\label{sec.curvature}
%
Let $(M,g)$ be an $n$-dimensional semi-Riemannian manifold and $\nabla$
a metric connection on $M$, possibly with torsion $T^\nabla$. We denote the
Levi-Civita connection on $M$ by $\nabla^g$. For a  vector field $V$
on $M$, we define a $1$-form  $A_V \in \Omega^1 (M, \End (TM))$  by 
\be\label{1}
A_V(X) Y \ :=\  g(X,Y) V - g(V,Y)X,
\ee
hence   $\nabla^g + A_V$ becomes a connection with vectorial torsion on $M$. 
Then the following formulas relating  the curvatures of $\nabla$ 
and  $\nabla + A_V$ hold:
\begin{lem}
The curvature quantities of the connections $\nabla$ and $\nabla+A_V$
satisfy the following relations:
\begin{enumerate}
\item Curvature transformation:
\begin{eqnarray*}
\kr^{\nabla + A_V} (X,Y) Z 
&=& \kr^\nabla (X,Y) Z + g(Y,Z) \nabla_XV - g(X,Z) \nabla_Y  V\\  
&& +\, \big[g(\nabla_YV,Z) - g(Y,Z) \|V\|^2 + g(V,Z) g (V,Y)\big] X \\
& & -\, \big[g(\nabla_XV,Z) - g(X,Z) \|V\|^2 + g(V,Z) g (V,X)\big] Y\\ 
&& +\, \big[g(Y,Z) g(V,X) - g(X,Z)g(V,Y)\big] V \\
&& + g(T^\nabla (X,Y),Z) V- g(Z,V) T^\nabla (X,Y).
\end{eqnarray*}
\item Ricci curvature:
\begin{eqnarray*}\label{Ricci} 
\Ric^{\nabla +A_V} (X,Y)  &=& \Ric^\nabla (X,Y) + [\Div^\nabla V + (2-n)
\|V\|^2]\, g(X,Y) \\
&& + (n-2) \, [ g (V,X) g(V,Y)  +  g(\nabla_XV,Y)] \\
&& + g(V,Y) \tr (X\haken T^\nabla)+ g(T^\nabla(V,X),Y).
\end{eqnarray*}
\item Scalar curvature:
\[
s^{\nabla + A_V} \ =\  s^\nabla +  2(n-1) 
\Div^\nabla V + (n-1) (2-n)\|V \|^2 + 2 \tr (V \haken T^\nabla).
\]
\end{enumerate}
\end{lem}
\begin{proof}
The formulas follow by a routine computation 
(for $\nabla = \nabla^g$, the terms involving $T^\nabla$
vanish, see also the Appendix of \cite{Agricola06}).
We normalize the Ricci curvature as follows
\[
\Ric^{ \nabla + A_V} (X,Y)\ = \ \sum_{i=1}^n 
\varepsilon_i g(\kr^{\nabla + A_V} (X,e_i)e_i,Y)
\]
for an orthonormal basis $e_1,\ldots, e_n $  ($g(e_i,e_j) =
\varepsilon_i\delta_{ij},\ \varepsilon_i = \pm 1$). The trace and 
divergence are defined by
\bdm
\tr (V\haken T^\nabla)\, :=\, \sum_{i=1}^n \varepsilon_{i} g  
(T^\nabla (V,e_i), e_i)),\quad
\Div^\nabla V \, := \,\sum_{i=1}^n \varepsilon_i g(\nabla_{e_i} V, e_i).
\qedhere
\edm 
\end{proof}
In case that $\nabla$ has skew symmetric torsion, 
$\Div^\nabla V = \Div^{\nabla^g} =: \Div V$ and $\tr (X\haken T^\nabla) = 0$.

Let us now consider the case where $\nabla =\nabla^g$ the Levi-Civita 
connection. We denote the curvature of $\nabla^g $ by $\kr^g$ and the curvature 
of $\nabla := \nabla^g+A_V$ by $\kr^V$, and analogously for the Ricci and 
scalar curvatures. We call a manifold $V$-flat if $\kr^V=0$ and use the words 
$V$-Einstein and so on in the same way.

\begin{cor}\label{cor.curv-general}
Let $(M,g)$ be an $n$-dimensional semi-Riemannian manifold, $\nabla$ a connection
with vectorial torsion $V$. Then
\begin{enumerate}
\item Ricci curvature:
\bdm\label{Ricci2} 
\Ric^{V} (X,Y)  = \Ric^g (X,Y) + [\Div\, V + (2-n)
\|V\|^2]\, g(X,Y) + (n-2) \, [ g (V,X) g(V,Y)  +  g(\nabla^g_XV,Y)].
\edm

\item Scalar curvature:
\[
s^{V} \ =\  s^g +  2(n-1) 
\Div\, V + (n-1) (2-n) \|V\|^2 .
\]
\end{enumerate}
\end{cor}
For surfaces, the formulas simplify 
\begin{cor}\label{cor.curv-surfaces}
Let $(M,g)$ be an $2$-dimensional semi-Riemannian manifold,
$\nabla$ a connection with vectorial torsion $V$. Then
\begin{enumerate}
\item Ricci curvature: $\Ric^{V} (X,Y)  = \Ric^g (X,Y) +\Div\, V g(X,Y)$
\item Scalar curvature: $s^{V} \ =\  s^g +  2 \,\Div V$
\end{enumerate}
In particular, a surface $M$ is $V$-Einstein if and only if $M$ is Einstein.
\end{cor}
\begin{NB}If $M$ is a closed surface,  the total $V$-scalar curvature 
$S^V(M):=\int\limits_M s^V d\mu $ is
equal to $2\pi\chi(M)$,  and therefore independent of $V$. Especially if $M$ is $V$-flat then $M$ is a torus.
\end{NB}

In dimension $n >2$  we consider the  the Schouten  tensor
\bdm
C^g = \frac{1}{n-2}( \frac{1}{2(n-1)}  s^g g - \Ric^g ) 
\kulk g
\edm
 and the Weyl tensor $W^g$ in Dimension $n>3$. Here $\kulk $ denotes the Kulkarni-Nomizu Product, which is usually defined for symmetric tensors, but we will also use it for two arbitrary   tensors:
\bdm
(\alpha \kulk \beta) (X,Y,Z,W) := \alpha (X,Z) \beta (Y, W) 
+ \alpha (Y,W) \beta (X,Z) - \alpha (X, W) \beta (Y,Z) 
- \alpha (Y,Z) \beta (X,W).
 \edm
Note that this is a tensor which is still  antisymmetric in the first 
as well as in the  last two variables.
We define the Schouten Tensor for the $V$-curvature in the same way as 
the usual Schouten Tensor: 
\bdm
C^V = \frac{1}{n-2}( \frac{1}{2(n-1)}  s^V g - \Ric^ V ) 
\kulk g
\edm
Then the equation for the $V$-curvature simplifies as follows:
\begin{prop}\label{Weyl-tensor}
Let $(M,g)$ be a semi-Riemannian manifold of dimension $n>2$,
$\nabla$ a connection with vectorial torsion $V$. Then for the $V$-curvature 
tensor holds
\begin{enumerate}
\item
$\kr^V = W^g +C^V$  if  $n>3$
\item
$\kr^V = C^V$ if $n=3$
\end{enumerate}
\end{prop}
\begin{proof}
Let $(e_1,\ldots,e_n)$ an orthonormal basis. We denote the components of the 
curvature tensore by 
$g( \kr (e_i, e_j) e_k, e_l) = \kr _{ijkl} $ and analogously for the 
other tensors. We always suppose that $i\neq j \neq k \neq l$. 
Note that both sides of the equations are pairwise antisymmetric  in the 
first two and in the  last two arguments. Therefore it suffices to 
calculate the following cases:
\begin{eqnarray*}
\kr^{V}_{ijji}  
&=& \kr^g_{ijji} + g( \nabla_{e_i} V, e_i ) + g( \nabla_{e_j}  V,e_j) - \|V\|^2 +  V_i^2 + V_j^2  \\  
&=&  \kr^g_{ijji}   + \frac{1}{n-2}( \Ric^V_{ii}-  \Ric^g_{ii} +   \Ric^V_{jj} -  \Ric^g_{jj}  
+ \frac{s^g -s^V}{n-1})    \\
&=&\kr^{g}_{ijji}  - C^g_{ijji}+C^V_{ijji}
\ =\  W^g_ {ijji} + C^V_{ijji}\\ 
\kr^{V}_{ijjk}  
&=& \kr^g_{ijjk} + g( \nabla_{e_i} V, e_k ) +  V_i V_k  
\ =\  \kr^g_{ijjk}   + \frac{1}{n-2}( \Ric^V_{ik}-  \Ric^g_{ik} )    \\
& =& W^g_ {ijjk} + C^V_{ijjk}\\ 
\kr^{V}_{ijkl}  
&=& \kr^g_{ijkl} =  W^g_ {ijkl} + C^V_{ijkl}
\end{eqnarray*}
Here we  inserted the formulas for the Ricci curvature of Corollary 
\ref{cor.curv-general} in the formula for the curvature tensor. 
Obviously, for  $n= 3$ the last equation as well as the Weyl tensor vanish.
\end{proof} 
%
\begin{cor}
Let $(M,g)$ a Riemannian manifold of Dimension $n>3$,
$V$ a vector field. If $M$ is $V$-flat, then $M$ is also conformally flat.
Conversely, if $M$ is conformally flat and $V$-Ricci flat, then $M$ is $V$-flat.
\end{cor}
\begin{NB}
The triple $(M,g,V)$ defines a Weyl structure, i.\,e.~a conformal 
manifold with a torsion free connection which is compatible with the conformal 
class.  The corresponding Weyl curvature  satisfies 
$\kr^{Weyl} = \kr ^V + dV^\flat $, therefore it coincides with the 
$V$-curvature if $dV^\flat = 0$ \cite{Gau}. This was the argument used in
\cite{AgFr10} to prove the first part of the  corollary if $dV^\flat = 0$.
\end{NB}  
Our aim is to show now that indeed, assuming the vector field $V$ to be closed
is a very natural condition.
The $V$-Ricci curvature for surfaces is symmetric, but in higher dimensions 
this is not necessarily the case. For reference, let us recall
that the first Bianchi identity for a metric connection with vectorial
torsion $V$ is (see for example \cite{Agricola06})
\be\label{eq.Bianchi-I}
\cyclic{X,Y,Z} \kr^V(X,Y)Z\ =\ \cyclic{X,Y,Z}dV^\flat (X,Y)Z.
\ee
\begin{prop}
Let $(M,g)$ be a semi-Riemannian manifold of dimension $n>2$,
$\nabla$ a connection with vectorial torsion $V$. Then the
following conditions are equivalent:
\begin{enumerate}
\item
The $V$-curvature $\kr^V(X,Y,Z,W)$ is symmetric with respect to the pairwise
exchange of $(X,Y)$ and $(Z,W)$,
\item the $V$-Ricci tensor $\Ric^V$ is symmetric, 
\item the $1$-form $V^\flat$ is  closed.
\end{enumerate}
\end{prop}
\begin{proof}
We proceed as in the Riemannian case: we take the inner product
of the first Bianchi identity with some fourth vector field $W$ and
sum cyclically over $X,Y,Z$, and $W$. This yields on the left
hand side
\bdm
\cyclic{X,Y,Z,W}\big(\cyclic{X,Y,Z} \kr^V(X,Y,Z,W)\big)\ =\ 
2 \kr^V(Z,X,Y,W)-2\kr^V(Y,W,Z,X),
\edm
hence the $V$-curvature is  `pairwise' symmetric if this quantity vanishes
identically. For the right hand side, we compute
\bdm
\cyclic{X,Y,Z,W}\big(\cyclic{X,Y,Z} dV^\flat(X,Y)g(Z,W)\big)\ =\ 
2g(\cyclic{X,Y,Z} dV^\flat(X,Y)Z,W).
\edm
Since the inner product is non-degenerate and $X,Y,Z$ can be chosen to be
linearly independent, we conclude that
the right hand side vanishes if and only if $dV^\flat=0$.

Of course, $\Ric^V$ will be symmetric if $\kr^V$ is symmetric; but let us prove
that this is indeed the only situation where this happens. By Corollary 
\ref{cor.curv-general},
$\Ric^V$ is symmetric if, for all vector fields $X,Y$ on $M$,
$g(\nabla_X^g V, Y) = g(\nabla_Y^g V,X)$ holds.
For the $1$-form $V^\flat$ corresponding to $V$, the Cartan differential is 
calculated as $d V^\flat (X,Y)= g(\nabla^g_XV,Y) - g(\nabla^g_YV,X)$. 
This  gives the result.
\end{proof}
\begin{NB}
A metric connection with skew torsion $T\in\Lambda^3(M)$ has symmetric
curvature if $T$ is parallel, $\nabla T=0$ \cite[p.32]{Agricola06}. 
This is again a hint that $dV^\flat=0$ is a more natural geometric condition
than $\nabla V=0$.
\end{NB}
\begin{dfn}
Let $\nabla$ be a metric connection with vectorial torsion $V$.
We agree to call $\nabla$ a connection with \emph{closed} vectorial torsion
if, in addition,  $dV^\flat=0$ holds. More generally, we will call a vector 
field \emph{closed} if its dual $1$-form is closed.
\end{dfn}
A second geometric interpretation of closed vectorial torsion is given 
in the following proposition.
\begin{prop}\label{prop.geom-sym-Ricci}
Let $(M,g)$ be a semi-Riemannian manifold. Suppose that
 $V$ is a vector field on $M$ with $g(V,V) = \| V\|^2 \neq 0$ everywhere, 
and $\nabla$ the metric connection with vectorial torsion $V$.
If $V$ is closed, the distribution 
$D = V^\bot$ is involutive and for every vector field $X$ in $D$ holds:
\be\label{3}
X(\|V\|^2) \ =\  g(V,[X,V]).
\ee
Conversely, if $D$ is an involutive $(n-1)$-dimensional  distribution on 
$(M,g)$ and $V$ a vector field of nowhere vanishing length orthogonal 
to $D$ such  that $V\haken dV^\flat=0$ holds, then the $V$-curvature
is symmetric. In this case,  the $V$-Ricci tensor satisfies 
for $\xi, \eta\in D$:
\begin{eqnarray*}
\Ric^V (\xi,\eta) - \Ric^g (\xi,\eta)\ \ \, &=& \left[
H \|V\| + \frac{1}{|V|} (V(\|V\|)) + (2-n) \|V\|^2\right] 
g(\xi,\eta)
 + (n-2) \|V\| \II  (\xi,\eta) , \\
\Ric^V (\xi,V) - \Ric^g(\xi,V) \ &=& \frac{n-2}{2}  \xi (\|V \|^2), \\
\Ric^V (V,V) - \Ric^g (V,V) &=& \|V\|^2 ( H \|V\|+ (n-1) \frac{1}{\|V\|} 
V (\|V\|)),
\end{eqnarray*}
where $\II$ is the second  fundamental  form, $H=\tr\, \II$ the mean 
curvature on the leaves of $D$, and $\Ric^g$ the  Ricci curvature 
of the Levi-Civita connection on $M$. 
\end{prop}
\begin{proof}
 For $X,Y \in V^\bot$ from $d V^\flat =0$ follows
\[
0 = g (V,\nabla^g_X Y -\nabla_Y^g X) = g(V,[X,Y]),
\]
which means that $[X,Y] \in V^\bot$ and so $V^\bot$ is involutive and 
(\ref{3}) is equivalent to $V \haken d V^\flat = 0$.
Conversely, if $V^\bot$ is involutive then for $X,Y\in V^\bot $
\[
dV^\flat (X,Y)= g (V,\nabla^g_X Y -\nabla_Y^g X) = g(V,[X,Y]) = 0.
\]
Together with $V \haken dV^\flat =0$ follows the symmetry of the $V$-Ricci 
curvature.
\end{proof}
\begin{NB}
By now, manifolds admitting a homogeneous structure
of strict type $TM\oplus \mathcal{T}'$ (see Proposition \ref{classessum})
are called \emph{cyclic homogeneous manifolds}.
Pastore and Verroca proved in \cite{Pastore&V91} that  
cyclic homogeneous manifolds whose
vectorial torsion part satisfies $dV^\flat=0$ are foliated by isoparametric
$(n-1)$-dimensional submanifolds---the assumption that $\kr^V$ is 
$\nabla$-parallel is crucial for this stronger result. 
It is known that cyclic homogeneous manifolds are never compact
\cite{Tricerri&V2}, which fits into the general picture (see the examples) 
below that all interesting examples are non compact. For dimension
$n\leq 4$, they are classified in \cite{Gadea&G&O14}. Examples on Lie groups
can be found in \cite{Gadea&G&O15}.
\end{NB}

\begin{NB}
For any metric connection $\nabla$ with torsion $T$, the differential
of a $1$-form $\omega$ satisfies
\bdm
d\omega (X,Y)\ =\ \nabla_{X}\omega(Y)-\nabla_Y\omega(X)+\omega(T(X,Y)).
\edm
Formula (\ref{eq.torsion}) implies that the last term vanishes for 
a connection with vectorial torsion and $\omega=V^\flat$ , hence we obtain
the remarkable identity
\be\label{eq.dV}
dV^\flat(X,Y)\ =\ \nabla_X V^\flat(Y)- \nabla_Y V^\flat(X)\ =\ 
g(\nabla_X V, Y) - g(\nabla_Y V,X),
\ee
in complete analogy to the classical formula expressing $dV^\flat$
through $\nabla^g V^\flat$. In particular, $V$ is closed if and only if
the tensor $S(X,Y):= g(\nabla_X V, Y)$ is symmetric in $X$ and $Y$.
\end{NB}
We now consider the  $V$-curvature for  some special vector fields.

\begin{exa}
If $V$ is a  Killing vector field, the  condition that $V$ be closed implies 
\[
g(\nabla_X^g V,Y) = 0
\]
for all vector fields $X,Y$ on $M$ and therefore  $V$ is a Riemannian
parallel vector field. In this case   the second fundamental form $\II$ 
of the orthogonal  distribution is zero. 
Then $\Ric^{V} = 0$ iff $\Ric^g (X,Y)  =(n-2) \|V\|^2  g(X,Y) $ for $X$ 
and $Y$ in the orthogonal distribution and $\Ric^g (V,X)  = 0$ for any 
vector  $X$. Therefore if $ M$ is simply connected, it is the product 
of a Einstein space of positiv scalar curvature  and $\R$.
 Moreover if $M$ is $V$-flat, it is the product of a sphere and $\R$.
\end{exa}
\begin{exa}
If the vector field $V$ is $\nabla$-parallel, it is closed by equation
(\ref{eq.dV}) and $\Div (V)  = (n-1) g(V,V) $ is constant.
If one integrates this identity over a closed Riemannian manifold, 
$V=0$ follows, proving that this condition is rather restrictive.
The curvature formulas are particularly simple, 
\begin{enumerate}
\item Ricci curvature: $\Ric^{V} (X,Y)  = \Ric^g (X,Y) +(n-1) g( V,V) g(X,Y)$
\item Scalar curvature: $s^{V} \ =\  s^g +  n(n-1) g(V,V)$
\end{enumerate}
Therefore, $M$ is $V$-Ricci flat exactly if $M$ is Einstein with 
Ricci curvature 
$-(n-1)g(V,V)g$ which is, up to a constant, the  Ricci curvature of 
hyperbolic space. Indeed, if $M $ is the hyperbolic space noted as 
warped product $\R \times _{e^t} \R^n$, then 
$V = \del_t$ is $V$-parallel and  $\Ric^{V} =0 $.
Actually hyperbolic space  satisfies even $\kr ^V = 0$ and according to  Tricerri and Vanhecke it is locally the only $V$-flat space with $\nabla V = 0$

\end{exa}
\begin{exa}
If $V$ is a closed conformal vector field,  $\nabla_X^g V = \lambda X$ for
all  vector fields $X$ on $M$ and a function $\lambda$ on $M$. Thus the 
leaves of the distribution $V^\bot$ outside the set  of  zeroes of $V$  are 
umbilic with mean curvature $\frac{n-1}{\|V\|}\lambda $ and
\[
(\Ric^{V} - \Ric^g) (X,Y) = (2(n-1)\lambda +\|V\|^2 (2-n)) g(X,Y) 
+ (n-2)g(V,X)g(V,Y)
\]
and
\[s^V-s^g = 2n(n-1)\lambda - (n-1)(n-2)\|V\|^2 .
\]
\end{exa}
Now we consider the case that the orthogonal distribution is still umbilic, 
but the vector field is not necessarily conformal. Umbilic distributions are
precisely $\mathrm{SO}(n-1)$-structures of vectorial type in the sense of
\cite{Agri&F06}.
\begin{prop}
Let $M$ be an $n$-dimensional  Riemannian manifold with an involutive 
$(n-1)$-dimensional  $D$, whose leaves are totally umbilic with 
mean curvature $H = (n-1)\lambda $, where $\lambda \in C^\infty(M)$. 
Let $N$  a normal unit field on the  leaves of $D$.

For $n  >2$,  a vector field $V$ is orthogonal to the distribution $D$ and closed
if and only if  $V = hN$ with $h \in C^\infty (M)$ and $h$ satisfies the equation
\be
d h (\xi) = h  g (N,[\xi ,N])
\label{5}\ee
for all vector fields $\xi$ in the distribution $D$.
 Then for all vector fields $\xi, \eta$ in $D$
\begin{eqnarray*}
\Ric^{V} (\xi,\eta) - \Ric^g (\xi,\eta) &=& ((2n-3) h\cdot\lambda 
+ (2-n)h^2 +N(h)) g(\xi,\eta)\\
\Ric^{V} (N,\xi) - \Ric^g (N,\xi) &=& (n-2) \xi (h) \\
\Ric^{V} (N,N) - \Ric^{g} (N,N) &=& (n-1) (h\lambda + N(h))
\end{eqnarray*}
\end{prop}
\begin{proof} 
Since $D$ is umbilic, $\nabla^g_\xi N = \lambda \xi $ for $\xi \in D.$
For $V = hN$, we obtain
\[
\nabla_\xi^g V =d h (\xi ) N  +h \lambda \xi.
\]
Therefore, $V \haken dV^\flat = 0$ if and only if (\ref{5})  holds. 
Then $\nabla_N^g V = \grad h$ because
\[
g(\nabla_N^gV,\xi ) = - g(V,\nabla_N^g\xi ) = - g(V,\nabla_\xi^g N) 
- g(V,[N,\xi]) = dh (\xi)
\]
according to (\ref{5}) and $g(\nabla_N^gV,N) = dh(N)$.
\end{proof}
From Hodge theory it is well known that the space of closed vector fields on a closed Riemannian manifold
is the orthogonal sum of the space of harmonic vector fields and the space 
of exact vector fields. Let us first consider the case that $V$ is  a 
gradient vector field. In \cite{Agricola&T04}, it has been shown that in 
this case the $V$-geodesics coincide with the geodesics of a conformally 
equivalent metric up to reparametrisation. Here we show that   the 
$V$-curvature coincides up to rescaling with the Riemannian
curvature of a conformally equivalent manifold. From the point of view 
of Weyl geometry this follows from the fact that
the Weyl structures  $(W, e^{2f}g, V-df)$ are all equivalent and have 
therefore the same Weyl  curvature. Alternatively, one compares
the definition of $A_V$  with the formulas for the connection of 
the conformally changed metric \cite{Besse} and concludes  that 
$\nabla^{\tilde g}= \nabla^V +B$ with $B(X) =- g(V,X) \Id  $; the following
formulas are then obtained by a straightforward calculation.
\begin{prop} \label{prop.conf. Aenderung}
Let $M$ be a manifold with metric $g$, $V= - \grad f$ and $\tilde g = e^{2f}g$. 
Then
\bdm
\kr^V= e^{-2f}\kr^{\tilde g},\quad
\Ric^V =\Ric^{\tilde g},\quad s^V=e^{2f}s^{\tilde g}.
\edm
In particular, $M$ is  $(-\grad f)$-flat (resp.\,$(-\grad f)$-Einstein 
resp.\,$(-\grad f)$-Ricci flat) if and only if $(M, \tilde g)$ is 
flat  (resp. Einstein resp.\,Ricci flat).
\end{prop}
%
%
\begin{exa}\label{exa.warped.product}
Let $(F, g_F)$ an $(n-1)$-dimensional Riemannian manifold and  
$f \in C^\infty (\R, \R_+)$. For $\eps  = \pm 1$
we consider on the differentiable manifold $M = \R \times F$ the
product metric $ g_\eps = \eps dt^2 + g_F$ and the warped product metric 
$\tilde  g_\eps = \eps dt^2 + f^2 g_F$; hence, the sign $\eps$
is introduced to cover Riemannian and Lorentzian signature in one expression.
In these examples the  distribution $\R \times TF$ has totally umbilic 
leaves $\{ t\} \times F$, with  mean curvature  
$( n-1) \lambda = (n-1) \frac{\dot f}{f}$ for the warped product.  

For the Ricci curvature of the Levi-Civita connection on $(M,\tilde g_{\pm 1})$,
it is well known that
\begin{eqnarray}
\Ric^{\tilde g_\eps} (\partial_t,\partial_t) &=& - \frac{(n-1)}{f} \ddot f \nonumber\\
\Ric^{\tilde g_\eps}  (\xi,\partial_t) &=& 0 \label{6}\\
\Ric^{\tilde g_\eps} &=& \Ric^{g_F} (\xi,\eta)- \epsilon  (\frac{\ddot f}{f} 
+ (n-2)\frac{\dot f^2}{f^2}) g(\xi,\eta)\nonumber
\end{eqnarray}
where $\xi, \eta$ are vector fields on $F$.

We are interested in connections 
$\widetilde \nabla^V := \nabla ^{\tilde g _ \eps}+ A_V$ 
with symmetric Ricci tensor  $\widetilde \Ric ^V_\eps$ and orthogonal 
distribution 
given by $\R \times TF$.
By equation (\ref{5}) this amounts to choosing $V = h\cdot \partial _t$
for some $h \in C^\infty (\R)$, because  for the normal unit vector field 
$N = \partial_t$, we have   
$g(\partial_t , [\partial_t,\xi]) = 0$ for vector fields $\xi$ in the leaves. 
If we ask moreover for a factor $h$ such that 
\[ 
\widetilde \Ric^{V}_\eps = \Ric^{g_{\eps}},
\]
the only possible choice is $h = \eps \frac{\dot f}{f}$, because 
$(M, \tilde g_\eps)$
and $(M,g_\eps)$ are conformally equivalent with the  conformal 
factor $f^{-2}$.
By Proposition \ref{prop.conf. Aenderung}, 
the  vector 
field we were looking for is thus
$V =\grad (\ln f)= \eps \frac{\dot f }{f} \del_t$.
This  is exactly the choice of vector field that 
appeared in \cite{Agricola&T04}.
\end{exa}
Although we are mainly interested in connections with pure vectorial 
torsion, it 
is straightforward to generalize our results to connections with an 
additional 
part in the torsion.  In \cite{Hehl}, the authors constructed
on the warped product $M \times_f  S^3$ a Lorentzian spacetime with nontrivial 
Nieh-Yan $4$-forms with nonvanishing torsion, but zero Ricci curvature. 
We generalize this construction in the following proposition.
To this purpose we  replace in  Example \ref{exa.warped.product}
 the Levi-Civita connection by a connection with
 skew symmetric torsion induced by a connection on $F$.  More exactly,    
let $\nabla^F$ be the Levi-Civita connection on  $F$ and $\nabla^F + \omega$ 
a connection on $F$ with skew symmetric torsion, and $\pi : M \to F$ the 
projection.  Then 
$\widetilde \nabla^\omega = \nabla^{ \tilde g_\eps} + \pi^* \omega$ is 
a metric connection on $( M,  \tilde g_\eps)$ with skew symmetric torsion. In 
\cite{ Hehl},  $ \omega $ is chosen such that the Ricci curvature of 
$\nabla^{g_{-1}} + \pi^* \omega$ vanishes. 
\begin{prop} 
Denote by $\widetilde \Ric^{\omega,V}_ \eps$ the Ricci curvature of 
$\widetilde\nabla^\omega + A_V$  on $(M,\tilde g_\eps)$ and by 
$\Ric^{ \omega }_\eps $ 
the Riemannian Ricci curvature of $\nabla^{g_\eps}+\pi^* \omega$ on 
the product $(M,g_\eps)$. Then
 \[
\widetilde \Ric ^{\omega,V}_\eps = \Ric^{ \omega }_\eps \text{ if and only if } 
V= \eps \frac{\dot f}{f} \partial_t.
\]
\end{prop} 
\begin{proof} 
We generalize Example  \ref{exa.warped.product}.
If $\Ric^{(F,\omega)}$ denotes the Ricci tensor of the connection 
$\nabla^F + \omega$ on $F$, then according to the formulas for the curvature 
of connection with skew symmetric torsion \cite{Agricola06} the Ricci tensor 
of $\widetilde \nabla ^\omega$ on $M \times_f F$ is the same as
(\ref{6}) with $\Ric^{g_F}$ replaced by $\Ric^{(F,\omega)}$.
 Because $V \haken T = 0$ and
$T$ is 
skew symmetric, the last two terms in (\ref{Ricci}) vanish in this situation, 
$\Div^{\nabla^\omega } V = \Div V$ and $\widetilde \nabla ^\omega  V 
=\widetilde  \nabla^g V$ for $V = h\partial_t$. 
Therefore the same argument  as in  Example \ref{exa.warped.product}  
shows the result.
\end{proof}
\begin{NB}
The importance of this result stems from the fact that it implies, 
in particular,
that it is possible to construct connections with vectorial
torsion on warped products of arbitrary dimension matching a given
Riemannian or Lorentzian curvature---for example, a Ricci-flat connection
with vectorial torsion in dimension $4$. This is totally different
from the case of skew torsion, where it was proved that
a Riemannian Einstein manifold can never be Einstein with skew torsion
if $n=4,5$ \cite[Exa 2.14]{Agricola&Fe14}. 
\end{NB}
Now we consider the case of a harmonic vector field $V$, which means a 
closed  vector field with vanishing divergence. For surfaces, it follows from 
Corollary 3.2  that  the $V$-Ricci curvature 
coincides with the Riemannian Ricci curvature. 

In the case of a non trivial 
spacelike vector field with vanishing divergence  for $n>2$
 \[
s^V = s^g-(n-1)(n-2)\|V \|^2 \leq s^g
\] 
 holds. But if $M$ is  closed manifold we can always choose a conformal 
equivalent metric such that the corresponding vectorfield is divergence 
free. More precisely for $n>2$  in \cite{Gau} has  been shown that on a 
closed Weyl manifold $(M,g,V) $ there exists a \emph{standard metric }  
$\tilde  g = e^{2f} g $ such that $\tilde V = V - df$  is divergence 
free related to the metric $\tilde g$ .
\begin{cor}
For any  vector field  $V$ on a closed Riemannian manifold $M$ of 
dimension $n>2$, the inequality
$e^{2f}s^{\tilde g}\geq s^V$ holds,  where $\tilde g = e^{2f}g$ is the 
standard metric for the corresponding Weyl structure.
\end{cor}
\begin{proof}
From Corollary \ref{cor.curv-general} and proposition \ref{prop.conf. Aenderung}, we conclude that  
$s^V = e^{2f}s^{\tilde V}  \leq  e^{2f}s^{\tilde g} $ where $ \tilde V  = V - df $ .
\end{proof}

%

%
\begin{exa}
%
In this example, we do not assume a priori that
the $V$-Ricci curvature is symmetric.  
Consider a $S^1$-fiber bundle $M \to B$ with $S^1$-invariant metric on 
$M$ and a vertical vector field $V$ on $M$ and the connection on $M$ 
induced by the metric.
If $V$ is a fundamental vector field on $M$ given by the $S^1$-action, 
$\xi$ is a vector field on $B$ and $\bar\xi$ its horizontal lift to $M$, 
then $[V,\bar\xi] = 0$. Therefore 
\bdm
\Div\, V \, =\, 0, \quad d V^{\flat}(\bar \xi, \bar \eta ) 
\,=\, 2 g(\nabla _{\bar \xi}V,\bar \eta)\  \text{ and } d V^{\flat}(V, \bar \xi) 
\,=\, -\bar \xi \|V\|^2.
\edm
For the curvature quantities this implies:
\begin{eqnarray*}
\Ric^{V} (\xi,\eta) - \Ric^g (\xi,\eta) 
&=& (n-2)(\frac{1}{2}dV^{\flat}(\bar\xi,\bar \eta)-\|V\|^2
 g(\xi,\eta)) \\
\Ric^{V} (V,\xi) - \Ric^g (V,\xi) &=& - \frac{(n-2)}{2} d V^{\flat}(V,\bar\xi) \\
\Ric^{V} (V,V)- \Ric^g(V,V) &=&0\\
s^V-s^g&=&-(n-1)(n-2)\|V\|^2
\end{eqnarray*}
Now let us study what happens when the $V$-Ricci tensor is 
symmetric, i.\,e.~$dV^\flat=0$.
The equivalent geometric criterion (\ref{3}) from Proposition 
\ref{prop.geom-sym-Ricci} is fulfilled if and only if 
$\bar\xi (\|V\|) = 0$  for all horizontal vector fields $\xi$  and a 
fundamental vector field $V$; because of the $S^1$-invariance of the 
metric this is equivalent to  $\|V\| = $const. Thus,
the $V$-Ricci tensor is symmetric if and only if $V$ has constant length.
Since Proposition \ref{prop.geom-sym-Ricci} also shows that then situation
the horizontal distribution is involutive, we can furthermore conclude that
the connection of the bundle is flat. By a well-known result of algebraic 
topology, this is equivalent to the vanishing of the real Euler class of
the $S^1$-bundle.
\end{exa}
\begin{prop}\label{prop.closed ricci flat}
Let $V$ be a vector field with vanishing divergence on a closed 
Riemannian manifold. Then
\bdm
 \int _M \Ric ^V (V,V) d \mu =  \int _M \Ric ^g (V,V) d \mu .
\edm
Moreover if $V$ is harmonic and $\Ric^V = 0$ then $V$ is 
$\nabla^g$-parallel and 
therefore the universal cover of $M$ is the product of an Einstein 
space of positive scalar curvature and $\R$.
\end{prop} 
\begin{proof}
The formula of Corollary \ref{cor.curv-general} yields for a vector 
field with vanishing divergence 
$
\Ric ^V (V,V) =  \Ric ^g (V,V) + \frac{n-2}{2} V (\| V\| )^2
$
and  the first result follows from Stokes' theorem  by integration. 
If $V$ is harmonic then the Bochner formula  yields for a $V$-Ricci 
flat manifold
\bdm
0 =  \int _M  (\Ric ^g (V,V) +\|\nabla^g V \|^2)  d \mu 
= \int _M\|\nabla^g V \|^2  d \mu. \qedhere
\edm
\end{proof}

\section{Dirac operators of connections with vectorial torsion}
%
Already in 1979, Thomas Friedrich observed that the Dirac operator $D$ 
associated
with a metric connection $\nabla$ with torsion $T$ has different properties 
depending on
the torsion type.  We summarize the
result of \cite{Friedrich79} in the following table:

\smallskip
\bdm
\ba{|c|c|c|c|} \hline
 & T\in TM & T\in \Lambda^3(M) &  T\in \mathcal{T}'\\[1mm] \hline
D\text{ formally self-adjoint} & \text{no} & \text{yes} & \text{yes}\\[1mm]\hline
D=D^g & \text{no} & \text{no} & \text{yes}\\[1mm]\hline
\ea
\edm

\bigskip
Assume from now on that $(M,g)$ is spin, and denote by $\Sigma M$ its spinor 
bundle.
For the connection $\nabla^g+ A_V(X)Y$ defined before, the lift to $\Sigma M$
is given by \cite[p.18]{Agricola06}
\begin{equation}
\nabla_X\psi\ =\ \nabla^g_X\psi+ \frac12  (X\wedge V) \cdot \psi.\label{con}
\end{equation}
One then easily computes that
\begin{equation}
D\psi\ =\ D^g\psi - \frac{n-1}{2}V\cdot \psi, \quad
D^*\psi\ =\ D^g\psi + \frac{n-1}{2}V\cdot \psi, \label{D}
\end{equation}
because the Clifford multiplication by a vector field is skew-adjoint with
respect to the hermitian product on the spin bundle. 
We start by constructing a few examples of 
manifolds with a $V$-parallel spinor field.
\begin{exa}
Let $F$ be an $(n-1)$-dimensional  manifold with Killing spinor with 
Killing number $ \frac{1}{2}$. Then $F$ is an Einstein space of scalar 
curvature $(n-1)(n-2)$. The induced spinor on $M = \R \times F$ is $ - \del_t$-parallel, because for  $n$ odd and $X$  tangent to $F$
\begin{eqnarray*}
\nabla ^g _{\del_t}\psi &=&0 \\
\nabla^g  _{X}\psi &=&\nabla^F  _{X}(\psi | F) =    \frac 12 X \cdot \del_t \cdot  \psi = \frac 12 X \wedge \del_t \cdot  \psi 
\end{eqnarray*}
For $n$ even, one performs  a similar calculation to obtain the result.  
Note that the manifold is $\del_t$- Ricci flat; in Theorem
 \ref{thm.par-spinors}, we will show that this is always true (with the 
additional assumption $dV^\flat=0$ in dimension $4$, which is satisfied 
in this example). 
\end{exa}
\begin{exa}
If $F$ is a manifold with a parallel spinor,  the induced spinor on the warped product
$\R \times  _f F$ is $\frac {\dot f}{f}\del_t$-parallel. Remember that in  
Example  \ref{exa.warped.product} we have shown that these manifolds 
are $-\frac {\dot f}{f}\del_t$-Ricci flat, as they should be.
\end{exa}
\begin{NB}
Since the connection is metric, $V$-parallel spinors have automatically
constant length, so we may assume the length to be one. 
If 
$V=\frac{1}{2}\grad u$ is a gradient vector field,   $V$-parallel 
spinor fields coincide with  weakly $T$-parallel spinors as defined in
\cite{Kim} for the choice of parameter  $\beta=\Id$. 
These spinors are used to construct solutions of the Einstein Dirac equation.
\end{NB}

\begin{prop}
 If $V$ is an exact vector field, the $V$-Dirac spectrum is the same as the  
spectrum of the Riemannian Dirac operator.
\end{prop}
\begin{proof}
By assumption, there exists a function $f$ such that  $V =\grad f$.
Choose the conformal factor  $h = e^{\frac{(n-1)}{2}f}$;  then, according to 
\cite[p.\,19]{BFGK},
\begin{equation}
D h  \psi \ = \ h D^g \psi + \grad h \psi -  \frac{n-1}{2}  \grad f h \psi 
\ = \ h D^g \psi.
\end{equation}
Therefore, if $\psi $ is an eigenspinor for $ D^g$, then $h \psi$ is an eigenspinor 
for $D$ and vice versa, because $h$ has no zeros. 
\end{proof}
\begin{NB}
 We recall the behaviour of spinor bundles under a conformal change of metric. 
If  $\tilde g = e^{2f }g$,  then there is an identification of the two spinor 
bundles which we denote by $\sim$. Comparing the formulas \cite[p.\,16]{BFGK} with 
(\ref{con}) shows that for the connection (\ref{con}) with $V = - \grad f$
\[
 \nabla^{\tilde g} _{\tilde X} \tilde \psi  = e^{-f}\widetilde{  \nabla_X  \psi }
\ \text{ and } \ 
 D^{ \tilde g} \tilde \psi  = e^{-f} \widetilde{ D \psi }.
\]
Therefore, $V$-parallel spinors and $V$-harmonic spinors can be identified with 
the parallel and harmonic spinors for the Levi-Civita  on the  corresponding 
conformally equivalent Riemannian manifold. 
\end{NB}
%

In \cite{Pfaeffle&S11}, the authors derived the following formula for $D_t^*D_t$,
\begin{equation}
D_t^*D_t\psi\ =\ \Delta^{\nabla}_{(n-1)t}\psi + \frac{1}{4}s^g\psi
+\frac{(n-1)t}{2}\mathrm{div}^g V \psi + t^2
\left(\frac{n-1}{2}\right)^2(2-n)\|V\|^2\psi \label{DD}
\end{equation}
where $\Delta^{\nabla}_t$ is the Laplacian associated with the connection
\begin{equation}
\nabla^t _X\psi\ =\ \nabla^g_X\psi+ \frac{t}{2}(X\wedge V)\cdot\psi \label{cov}
\end{equation}
and $D_t$  the corresponding Dirac operator. As for connections with skew 
torsion, one observes that a rescaling of the vector
field is necessary; but while this rescaling was by a constant in the skew
torsion case, it turns out to be dimension dependent here. In particular,
for $n=2$ the scaling factor is equal to one!
According to \cite{Pfaeffle&S11}, the Laplacian of the rescaled connection 
satisfies
\begin{equation}
\Delta^{\nabla }_{n-1} = (D^g)^2 - \frac 14 s^g + \frac{n-1}{2} [2V\cdot D^g 
+ 2 \nabla^g_V  - d(V^b)] + (\frac{n-1}{2})^2 (n-1) \|V\|^2 
\label{Lresc}
\end{equation}
and therefore the Laplacian $\Delta^{\nabla}_t$  associated with 
$\nabla^g + A_{tV}$ is given by
\begin{equation}
\Delta^{\nabla}_t = (D^g)^2 - \frac 14 s^g +t VD^g + t \nabla_V^g 
- t \frac 12 d(V^b) 
+ t^2 \frac 14 (n-1)\|V\|^2 \label{L}.
\end{equation}
\begin{prop}
Let $V$ a vector field on a spin manifold. Then for the  Dirac 
operator $D_t$  corresponding to the connection ($\ref{cov}$) we get:
\begin{equation}
D_t^*D_t\psi\ =\ \Delta^{\nabla}_t \psi + \frac{1}{4}s^g\psi 
+t \frac{n-1}{2}\mathrm{div}^g V \psi + t^2
\frac{(n-1)(n-2)}{4}\|V\|^2\psi +t
(n-2)(VD^g+ \nabla^g_V- \frac{1}{2}dV^\flat)\end{equation}
\end{prop}
\begin{proof}
We deduce from  (\ref{L}) 
\bdm
\Delta^{\nabla}_{(n-1)t} \psi\ =\   \Delta^{\nabla}_t \psi + t^2 n
\frac{(n-1)(n-2)}{4}\|V\|^2\psi + t (n-2)(VD^g+\nabla^g_V-\frac{1}{2}dV^\flat).
\edm
By replacing this expression in 
 (\ref{DD}) and using (\ref{D}), one obtains  the result.
\end{proof}
We shall now investigate what can be said about $V$-parallel spinor fields.
In Weyl geometry, there is a lift  of the Weyl connectection to  
weighted spinor bundles. $V$-parallel spinor fields can be identified with 
parallel spinors of weight zero for the Weyl connection. These have been 
studied by \cite{Mor}. The author showed the following proposition:
\begin{prop}\label{prop.mor}
Let $ (M,g,V)$ be a  Weyl manifold of dimension $n\geq 3$, and in addition compact if $n=4$. If it admits a non trivial parallel spinor of weight zero, 
then $V$ is closed. 
\end{prop}
The following proposition may 
already be found in
\cite[Thm 2.1.]{Agri&F06}, we reproduce a proof for completeness.
\begin{prop}\label{prop.sc.parallel}
If $\psi$ is a $V$-parallel spinor field, the following identities hold:
\begin{enumerate}
\item The Riemannian scalar curvature is given by
\bdm
s^g   = (n-1)(n-2) \|V\|^2 - 2(n-1) \Div V ,\ \text{ i.\,e. }\ s^V\, =\, 0.
\edm
\item The square of the Riemannian Dirac operator acts on $\psi$ by
\bdm
(D^g)^2 \psi = (\frac{n-1}{2} ||V||)^2\psi - \frac{n-1}{2} \Div V \psi  =\frac{1}{4}(s^g + (n-1) \| V \|^2)\psi.
\edm
\end{enumerate}
\end{prop}
\begin{proof}
For a $V$-parallel spinor, equation (\ref{con}) implies for $X =V$:
\bdm
\nabla^g_V \psi=0
\edm
and equation (\ref{D}) is reduced to $D^g \psi = \frac{n-1}{2} V\cdot  \psi$.
By applying the Riemannian identity for $D^g(V\cdot\psi)$ (\cite[p.\,19]{BFGK}),
we obtain
\begin{equation}\label{Dquadrat}
(D^g)^2 \psi = (\frac{n-1}{2} ||V||)^2\psi - \frac{n-1}{2} \Div V \psi 
+ \frac{n-1}{2} d(V^b)\psi,
\end{equation}
and from  (\ref{L}) we get
\begin{equation}\label{scaldV}
s^g \psi   =\ \left[(n-1)(n-2) ||V||^2 - 2(n-1) \Div V + 2 (n-2) d (V^b)\right]\psi.
\end{equation}
%
But for any $2$-form $\alpha$, the quantity  $(\alpha \psi, \psi)$ is 
purely imaginary, therefore the last term vanishes. This gives 
$ d(V^b)\psi = 0$---in fact, we already know that $ d(V^b)=0$ from Moroianu's
result except when $n=2,4$. Either way, inserting $ d(V^b)\psi = 0$  
in (\ref{Dquadrat}) und 
(\ref{scaldV}) yields the proposition after a short calculation. 
\end{proof}
%
 For a surface, we conclude that $V$ has to be closed, $s^g =  -  2\,\Div V$  
and Stokes' formula implies:
\begin{cor}
The only closed surface admitting non-trivial $V$-parallel spinor fields 
is the torus.
\end{cor}
The following example shows  that on the flat $2$-dimensional torus, 
there exists a vector field $V$ admitting 
$V$-parallel  spinor fields for every spin structure.
\begin{exa}\label{exa.torus}
 On the $2$-dimensional flat torus $T^2= [0,1]/\sim $ there are four spin 
structures, which we label with $( \eps_1, \eps_2)$ with $\eps_ i \in \{0,1\}$. 
The space of continuous spinors can be identified with
\bdm
\{\psi \in C^0 ( [0,1]^2, \C^2)\, \vert \, \psi (x,0) =(-1)^{ \eps _2} \psi (x,1)
\text{ and } \psi (0,y) = (-1)^{ \eps _1} \psi (1,y) \}.
\edm
Therefore, Example  \ref{exa.Rn} implies that  on $T^2$ there are 
$(v_1,v_2)$-parallel spinors  if and only if  $ v_1 = 2(k + \eps_2 )\pi $ and 
$ v_2 = 2(m + \eps_2 )\pi $ where $k,m \in \Z$.
\end{exa}
If there exists a $V$-parallel spinor field on $\R^2 $ then $ V = \grad f$ 
for a harmonic function $f $ anf therefore $V$-parallel spinor fields 
correspond to parallel spinor fields on the corresponding conformal 
equivalent manifold. We give a more concrete example.
\begin{exa}\label{exa.Rn}
On $\R^2$,  every constant vector field $V= (v_1,v_2)$ admits a $2$-dimensional
space of  $V$-parallel spinors, since  the spinor derivative is given by
\bdm
\nabla ^V _{\del_x}\psi\ =\  \del_x \psi + \frac 12 v_2  \omega \cdot \psi,\quad
\nabla^V  _{\del_y}\psi\ =\ \del_y \psi -  \frac 12 v_1  \omega \cdot \psi 
\edm
with 
$\omega = \begin{bmatrix} -i & 0\\ 0 & i \end{bmatrix}$.  
Therefore,  the space of  $V$-parallel spinors on $\R^2 $  is spanned  by 
$e^{ \frac{i}{ 2}(v_2 x-v_1y)} e_1 $ and $ e^{- \frac{i}{ 2}(v_2 x-v_1y)} e_2 $.
\end{exa}
\begin{NB}
 In \cite{Mor} parallel spinors of  Weyl structures  $(M,g, V)$ 
have been investigated, in particular in dimension $4$.
  In this case  the author shows that a compact Weyl $4$-manifold with
a $V$-parallel spinor is a hyperhermitian manifold and thus, by a result of 
Boyer, conformally equivalent to a torus, a K3 surface, or a Hopf surface.
\end{NB}


%
\begin{NB}
Suppose that $\psi$ is a $V$-parallel spinor field, $n>2$.
If $\Div V = 0$ then $M$ has constant positive scalar curvature
\bdm
s^g = (n-1)(n-2) \| V \|^2
\edm
and therefore the expression (2) of the previous Proposition is
reduced to
\bdm
(D^g)^2 \psi \ =\  (\frac{n-1}{2} ||V||)^2\psi\  =\ 
\frac{n-1}{4(n-2)}s^g\psi .
\edm
A $D^g$-eigenspinor is given by $\tilde\psi= \sqrt{\frac{(n-1)s^g}{4(n-2)}}
\,\psi + D^g\psi$.
\end{NB}

%
%
A deeper analysis will now show that as in the case of the Levi-Civita connection, 
the $V$-Ricci curvature of $M$ will vanish if it is symmetric and  if
$M$  admits a nontrivial $V$- parallel spinor.

Recall that the  curvature operator
of any spin connection can be understood as a endomorphism-valued $2$-form,
\bdm
\mathcal{R}(X,Y)\psi \ =\ \nabla_X\nabla_Y \psi - \nabla_Y\nabla_X\psi -
\nabla_{[X,Y]}\psi.
\edm
One checks that it is related to the curvature operator on $2$-forms 
defined through
\bdm
\kr(e_i\wedge e_j)\ :=\ \sum_{k<l} R_{ijkl}  \, e_k\wedge e_l
\edm
by the relation
\bdm
\kr (X,Y)\psi \ =\ \frac{1}{2}\, \mathcal{R}(X\wedge Y)\cdot\psi.
\edm
The following identity is crucial for deriving integrability conditions
for spinor fields satisfying first order differential equations.
It generalizes a well-known
result of Friedrich \cite[Satz 5.2]{Friedrich80} for the 
Levi-Civita connection ($V=0$).
\begin{thm} \label{thm.Ric-id}
%
Let $\nabla$ be a metric spin connection with vectorial torsion. 
Then, the following identity holds for any spinor field $\psi$ and any 
vector field $X$:
\[
\Ric ^V(X) \cdot \psi  = -2\sum_{k=1}^{n} e_k \mathcal{R}^V(X,e_k)\psi 
-  dV^\flat \wedge  X \cdot \psi.
\]
\end{thm}
\begin{proof}
Rewrite the first term on the right hand side (without the numerical factor) as
\bdm
\sum_{k=1}^n e_k \mathcal{R}^V(e_l,e_k)\ =\ \frac{1}{2}\sum_{k=1}^n e_k\cdot
\mathcal{R}^V(e_l\wedge e_k)\ =\ \frac{1}{2}\sum_{k=1}^n \sum_{i<j} R^V_{lkij}
e_k e_i e_j \ =:\ R_1 + R_2,
\edm
where $R_1$ denotes all terms with three different indices $k,i,j$, and $R_2$
all terms with at least one repeated index. We first discuss $R_1$:
\bdm
R_1 \ =\ \frac{1}{2} \sum_{i<j} \left[\sum_{k<i}R^V_{lkij}e_k e_i e_j+
\sum_{i<k<j}R^V_{lkij}e_k e_i e_j+ \sum_{j<k}R^V_{lkij}e_k e_i e_j \right]\ =\
\frac{1}{2} \sum_{k<i<j} \stackrel{k,i,j}{\mathfrak{S}} \, R^V_{lkij} e_k e_i e_j ,
\edm
where the symbol $\mathfrak{S}$ denotes the cyclic sum. The first Bianchi identity
for a metric connection with vectorial torsion (see equation (\ref{eq.Bianchi-I}))
implies then $R_1= - \frac 12 dV^\flat \wedge e_l^\flat $.
The second term does  not depend on the detailed type of the connection, so 
a similar argument as in \cite{Agr/BB/Kim} shows
\bdm
R_2 \ =\ -\frac{1}{2} \sum_{r=1}^n \left[ \sum_{p=1}^{r-1} R^c_{lppr}e_e+
\sum_{q=r+1}^n R^c_{lqqr}e_r\right].
\edm
But since the Ricci tensor is exactly the contraction of the curvature,
$R_2= - \Ric^V(e_l)/2$. This ends the proof.
\end{proof}
For closed vector fields $V$, Theorem \ref{thm.Ric-id} amounts therefore
to  an identity that looks formally as for $\nabla=\nabla^g$. 
\begin{thm}\label{thm.par-spinors}
Let $(M,g)$ be a semi-Riemannian spin manifold, $\nabla$ a connection
with vectorial torsion $V$, $n\geq 2$.
\begin{enumerate}
\item
If $M$ admits a nontrivial $V$-parallel spinor field,
 then $\Ric^V =0$ and $dV^\flat =0$ hold, and in particular,
$M$ is locally conformally Ricci flat. 
If $n=4$, $\Ric^V$ is totally
skew symmetric and given by    $\Ric^V(X) =X\haken dV^\flat$.
\item If $dV^\flat =0$, there are no nontrivial $\nabla$-Killing
spinor fields, i.\,e.~spinor fields satisfying 
the equation  $\nabla_X\psi=\beta\, X\cdot\psi$ for some $\beta \neq 0$.
\end{enumerate}
\end{thm}
\begin{proof}
For both claims, let's suppose $\psi$ is a spinor satisfying 
the equation  $\nabla_X\psi=\beta\, X\cdot\psi$ for all vector fields $X$
($\nabla$-parallel spinor fields will be obtained by choosing $\beta=0$).
Such a spinor field has automatically constant length, because it is
parallel for the metric spinorial connection 
$\tilde{\nabla}_X:=\nabla_X-\beta\, X \cdot$. Then
\begin{eqnarray*}
\kr(X,Y)\psi & =& \nabla_X\nabla_Y \psi - \nabla_Y\nabla_X\psi -
\nabla_{[X,Y]}\psi\\ 
& = & \nabla_X(\beta\, Y\cdot\psi) - \nabla_Y(\beta\, X\cdot\psi) -\beta
[X,Y]\cdot\psi\\
& = &\beta(\nabla_X Y -\nabla_Y X - [X,Y])\cdot\psi + \beta(Y\cdot \nabla_X \psi - X\cdot \nabla _Y \psi)\\
&=& \beta( g(V,X)Y - g(V,Y)X)\cdot\psi +\beta^2 (Y\cdot X - X\cdot Y)\cdot \psi\\
&=& \beta( g(V,X)Y - g(V,Y)X)\cdot\psi +2\, \beta^2 (Y\cdot X + g(X,Y)) \cdot\psi. 
\end{eqnarray*}
Therefore, the curvature contraction may be computed,
\bdm
\sum_{k=1}^n e_k\kr^V(X,e_k)\ =\ -\beta (n g(V,X)+V\cdot X)\cdot\psi
+ 2\beta^2(1-n)X\cdot\psi.
\edm
By Theorem \ref{thm.Ric-id}, we can conclude that the following equation is an
integrability condition for the existence of such a spinor field,
\be\label{eq.int-cond-KST}
\Ric^V(X)\cdot\psi\ =\ 2 \beta n\, g(V,X)\,\psi + 2\beta\, V\cdot X\cdot\psi
+ 4\beta^2(n-1)X\cdot\psi - (dV^\flat\wedge X)\cdot\psi.
\ee
We now discuss the two situations occuring in the statement. Let's treat the 
easier case first, i.\,e.~$\beta\neq 0$ and $dV^\flat=0$.
The last term of the integrability condition hence  vanishes. 
Let $(-,-)$ be the positive definite scalar product on spinor fields induced 
by the canonical hermitian product of the spinor  bundle, and recall
that it satisfies $(X\cdot\vphi,\vphi)=0$  for any spinor field $\vphi$.
We take the scalar product of the remaining identity (\ref{eq.int-cond-KST})
with $\psi$ and, by the previous remark, we are finally left with 
\bdm
0\ =\ \beta n\, g(V,X)\|\psi\|^2 + \beta\,(V\cdot X\cdot \psi,\psi)+0.
\edm
But $(V\cdot X\cdot \psi,\psi) = - (X\cdot \psi,V\cdot\psi)= -g(X,V)\|\psi\|^2 $,
so this identity cannot hold for $\beta\neq 0$ if $n\neq 1$.

Now let's consider the case of a  $V$-parallel spinor field, i.\,e.~$\beta=0$.
The integrability condition (\ref{eq.int-cond-KST}) is thus reduced to
\bdm
\Ric^V(X)\cdot\psi\ =\ - (dV^\flat\wedge X)\cdot\psi.
\edm
Inner and exterior product are related by 
$X\cdot \omega = X\wedge\omega- X\haken \omega$.
From Prosition \ref{prop.sc.parallel}, we know that $dV^\flat\cdot\psi=0$,
hence $X\cdot dV^\flat\cdot\psi=0 $ and thus 
\bdm
- (dV^\flat\wedge X)\cdot\psi\ =\ + (X\wedge dV^\flat )\cdot\psi\ =\ 
(X\haken dV^\flat)\cdot\psi.
\edm
Hence, the integrability condition is reduced to
$\Ric^V(X)\cdot\psi = (X\haken dV^\flat)\cdot\psi$. Since $\psi$ has constant length,
this implies 
\bdm
\Ric^V(X) \ =\ X\haken dV^\flat.
\edm
Viewed as an endomorphism, $X\haken dV^\flat$ is antisymmetric,
whereas $\Ric^V$ may be split into its symmetric and antisymmetric part according 
to Corollary \ref{cor.curv-general},
\bdm
\Ric^V(X) \ =\ \Ric^V_{\text{sym}}(X) + \frac{n-2}{2} X\haken dV^\flat.
\edm
We conclude $\Ric^V_{\text{sym}}(X)= \frac{4-n}{2} \,X\haken dV^\flat$.
One is symmetric, one is antisymmetric, so both have to vanish. If $n\neq 4$,
this implies $dV^\flat=0$ and then also $\Ric^V(X)=0$.
For $n=4$, we can only conclude $\Ric^V_{\text{sym}}(X)=0$, hence
 $\Ric^V$ has only antisymmetric part given by $\Ric^V(X) = X\haken dV^\flat$.
\end{proof}
\begin{NB}
The dimension distinction in statement (1) of the Corollary cannot be removed.
Indeed, in Section 7 of \cite{Mor}, an example of a non closed vector field $V$
admitting non-trivial $V$-parallel spinors is given
on an open subset of $\C^2$. Thus, $\Ric^V(X) = X\haken dV^\flat\neq 0$ and
statement (1) cannot be improved to $\Ric^V(X)=0$ for $n=4$ without 
the assumption of compactness (compare Proposition \ref{prop.mor}).
However, our result proves that $dV^\flat=0$ if and only if $\Ric^V=0$
in the non-compact case, and gives a formula expressing one quantity 
through the other. Oddly enough, we recover that $s^V=0$ even in dimension
$4$ (as it should be by Proposition \ref{prop.sc.parallel}), since
$\Ric^V$, although possibly non-zero, is always skew-symmetric and therefore
trace-free.
\end{NB}
Recall that  a  closed manifold of dimension  $n>2$ with vector field $V$ 
carries a conformally equivalent  metric  for which   the corresponding 
vector field has vanishing divergence.   Applying Proposition 
\ref{prop.closed ricci flat} thus yields
\begin{cor} 
Let $M$ be a closed  Riemannian  manifold of dimension $n>2$ with a 
$V$-parallel spinor. Then $M$  is  conformally equivalent either  
to a manifold with parallel spinor or to a manifold  whose universal 
cover is the product of $\R $ and an Einstein space of positive scalar 
curvature.
\end{cor}

%


%

    
\end{document}